\documentclass[a4paper,12pt]{amsart}
\makeatletter
\@namedef{subjclassname@2020}{%
\textup{2020} Mathematics Subject Classification}
\makeatother
\newtheorem{thm}{Theorem}

\newtheorem{prop}[thm]{Proposition}
\newtheorem{lemma}[thm]{Lemma}

\newtheorem{exam}[thm]{Example}

\renewcommand{\Im}{\operatorname{\rm{Im}}}
\renewcommand{\Re}{\operatorname{\rm{Re}}}
\def\B{\mathbb B}
\def\C{\mathbb C}
\def\R{\mathbb R}
\def\D{\mathbb D}
\def\L{\mathcal L}

\def\d{\delta}
\def\z{\zeta}

\begin{document}

\title[Growth of Sibony metric and Bergman kernel]
{Growth of Sibony metric and Bergman \\
kernel for domains with low regularity}

\author{Nikolai Nikolov and Pascal J.~Thomas}

\address{N. Nikolov\\
Institute of Mathematics and Informatics\\
Bulgarian Academy of Sciences\\
Acad. G. Bonchev Str., Block 8\\
1113 Sofia, Bulgaria}
\email{nik@math.bas.bg}
\address{Faculty of Information Sciences\\
State University of Library Studies
and Information Technologies\\
69A, Shipchenski prohod Str.\\
1574 Sofia, Bulgaria}

\address{P.J.~Thomas\\
Institut de Math\'ematiques de Toulouse; UMR5219 \\
Universit\'e de Toulouse; CNRS \\
UPS, F-31062 Toulouse Cedex 9, France} \email{pascal.thomas@math.univ-toulouse.fr}

\begin{abstract} It is shown that even a weak multidimensional
Suita conjecture fails for any bounded non-pseudoconvex domain with
$\mathcal C^1$ boundary: the product of the Bergman kernel by the volume of the
indicatrix of the Azukawa metric is not bounded below. This is obtained by finding a direction along
which the Sibony metric tends to infinity as the base point tends to the boundary.
The analogous statement fails for a Lipschitz boundary. For a general $\mathcal C^1$
boundary, we give estimates for the Sibony metric in terms of some directional
distance functions.
For bounded pseudoconvex domains, the Blocki-Zwonek Suita-type theorem
implies growth to infinity of the Bergman kernel; the fact that the Bergman kernel grows as the square of the reciprocal
of the distance to the boundary, proved by S. Fu in the $\mathcal C^2$ case, is
extended to bounded pseudoconvex domains with Lipschitz boun\-daries.
\end{abstract}

\thanks{The first named author is partially supported by the National Science Fund,
Bulgaria under contract DN 12/2.}

\subjclass[2020]{32F45}

\keywords{Suita conjecture, Bergman kernel, Azukawa and Sibony metrics}

\maketitle

\section{Results}
\label{main}
Let $D$ be a domain in $\C^n,$ $z\in D,$ $X\in\C^n.$ Define the Bergman kernel,
the Azukawa metric, and the Sibony metric, respectively, as follows (see e.g. \cite{JP}):
$$K_D(z):=\sup\{|f(z)|^2:f\in\mathcal O(D),\ ||f||_{L^2(D)}\leq 1\};$$
$$A_D(z;X):=\limsup_{\lambda\to 0}\frac{\exp g_D(z,z+\lambda X)}{|\lambda|},$$
where $g_D(z,w):=\sup\{u(w):u\in\mbox{PSH}(D),\ u<0,\ u<\log||\cdot-z||+C\}$
is the pluricomplex Green function of $D$ with pole at $z;$
$$S_D(z;X):=\sup_v[\L_v(z;X)]^{1/2},$$
where $\L_v$ is the Levi form of $v$, and the supremum is taken over all
functions $v:D\to[0,1)$ such that $v(z)=0$, $\log v$ is plurisubharmonic
on $D$, and $v$ is of class $\mathcal C^2$ near $z.$

Let $M_D\in\{A_D, S_D\}$ and $V^M_D(z)$ be the volume of the indicatrix
$$I^M_D(z):=\{X\in\C^n: M_D(z;X)<1\}.$$

Z. B\l ocki and W. Zwonek \cite[Theorem 2]{BZ} proved the following.

\begin{thm}\label{1} If $D$ is a pseudoconvex domain in $\C^n,$
then $$K_D(z)V^A_D(z)\ge 1,\quad z\in D.\footnote{If $V^A_D(z)=\infty,$
then $K_D(z)=0.$}$$
\end{thm}

Theorem \ref{1} for $n=1$ is known as the Suita conjecture (see \cite{Sui}).
The first proof of this conjecture was given in \cite{Blo}.

On the other hand, by \cite[Proposition 2]{Nik}, even a weaker version of Theorem 1
fails for bounded non-pseudoconvex domains with $C^{1+\varepsilon}$ boundaries. Our
first aim is to extend this result to $C^1$ boundaries.

\begin{prop}\label{2} Let $D$ is a bounded non-pseudoconvex domain in $\C^n$ with
$C^1$ boundary. Then there exists a sequence $(z_j)_j\subset D$ such that
$$\lim_{j\to\infty}K_D(z_j)V^A_D(z_j)=0.$$
\end{prop}

Since $D$ is non-pseudoconvex and $\partial D$ is of class $\mathcal C^1$,
there exists a point $p\in\partial D$ such that
$$\limsup_{z\to p} K_D(z)<\infty.$$

On the other hand, since $S_D\le A_D,$ then $I^A_D\subset I^S_D$ and hence $V^A_D \le V^S_D.$
So, Proposition \ref{2} will be a consequence of the following.

\begin{prop}\label{3} Let $D$ be a bounded domain $D$ in $\C^n$ with
$C^1$ boundary near $p\in\partial D.$ Then
$$\lim_{z\to p}V^S_D(z)=0.$$
\end{prop}

The proof will be given in Subsection \ref{pfp3}.

Combining Theorem \ref{1} and Proposition \ref{3}, it follows that
$$\lim_{z\to p}K_D(z)=\infty$$
if, in addition, $D$ is pseudoconvex. Proposition \ref{5} below says more.

On the other hand, Proposition \ref{3} may fail in the Lipschitz case even for $V^A_D.$

\begin{exam}\label{4} Let $D:=\{z\in\C^2:1<|z_1|+|z_2|<2\}.$ Then
$$\limsup_{x\to 1^+}A_D((x,0);X)<\infty$$
uniformly in the unit vectors $X.$ In particular,
$$\liminf_{x\to 1^+}V^A_D((x,0))>0.$$
\end{exam}

The proof follows that of \cite[Proposition 5]{DNT}
(for $S_D,$ see also the proof of \cite[Proposition 2]{FL}.)

\begin{prop}\label{5} Let $D$ be a bounded pseudoconvex domain $D$ in $\C^n$ with
Lipschitz boundary near $p\in\partial D.$ Then
$$\liminf_{z\to p}K_D(z)\d^2_D(z)>0.$$
\end{prop}

When the boundary is $C^2,$ Proposition \ref{5} is due to S. Fu
\cite[Theorem, p. 979]{Fu}. The proof for the Lipschitz case is given in Subsection \ref{pfberg}.

Lemma \ref{finv}, in Section \ref{proofs} below, and \cite[Proposition 5]{DNT} lead to the following generalization
of \cite[Theorem 2]{FL} and \cite[Corollary 6]{DNT}.

\begin{prop}
\label{modelsib}
Let $f:\mathbb R_+\to\mathbb R_+$ be a continuous function such that $f\not\equiv 0,$
$f(0)=0,$ and $f(t)/t$ is an increasing function for $t>0.$
Let $D$ be a bounded domain
in $\mathbb C^n$ given by $\Re (z_1)+O(|\Im (z_1)|)<f(||z'||)$
near $0\in\partial D.$ There exists a constant $c>1$ such that if
$\d>0$ is small enough and $X\in\mathbb C^n,$ then
$$
c^{-1}B(\delta;X)\le S_D(q_\delta;X)\le A_D(q_\delta;X)\le cB(\delta;X),
$$
where $q_\delta:=(-\delta,0')$ and
$\displaystyle B(\delta;X):=\frac{f^{-1}(\delta)}{\delta}|X_1|+||X||.$
\end{prop}

The hypothesis that $f(t)/t$ is  increasing is a sort of local concavity
of the domain $D$. We can remove this assumption for the lower bound.
Let $D$ be a domain in $\mathbb C^n$ with $C^1$ boundary near $p\in\partial D.$
For $q\in D$ near $p$ choose a point $p_q\in\partial D$ such that
$||q-p_q||=\delta_D(q)$. 
Let $L_q$ be the complex line through $q$ and $p_q$,
which contains the real line normal to $\partial D$ at $p_q$.
Let $H_{q^*}$ be the complex hyperplane through $q^*:=2p_q-q$ which is
orthogonal to $L_q.$ Set $\delta^*_D(q):=\min\left( 1, \delta_{\overline D\cap H_{q^*}}(q^*)\right)$
(the $1$ is there because the hyperplane $H_{q^*}$ could fail to intersect $\overline D$,
 for example, if $D$ is convex).

\begin{prop}
\label{compconv}
Let $D$ be a bounded domain in $\mathbb C^n$ with $C^1$ boundary near $p\in\partial D.$
There exists a constant $c>0$ such that
if $q\in D$ near $p$ and $X\in\mathbb C^n,$ then
$$
cS_D(q;X)\ge\frac{\delta^*_D(q)}{\delta_D(q)}||X_q||+||X||,
$$
where $X_q$ is the orthogonal projection of $X$ onto $L_q$. 
\end{prop}

The proofs of Propositions \ref{modelsib} and \ref{compconv} are given
in Subsections \ref{pfsib} and \ref{pfcomp}.

It would be interesting to know whether Proposition \ref{compconv} is still valid in the case
of a Lipschitz boundary.

\section{Proofs}
\label{proofs}

\subsection{Proof of Proposition \ref{3}.}
\label{pfp3}

By a linear change of coordinates, we may assume that $p=0$ and the outer normal to $\partial D$ at $p$
is $(1,0, \dots, 0)$.
There exists an open polydisk  $\mathcal U$
centered at the origin and included in the unit polydisk so that, with $z_1=:x_1+iy_1$ and $z':=(z_2, \dots , z_n)$,
\begin{equation}
\label{model}
D\cap \mathcal U = \mathcal U \cap \left\{ x_1 < f_0 (y_1, z') \right\},
\end{equation}
where $f_0$ is a real-valued $\mathcal C^1$ function such that
$f_0 (y_1, z') = o( (y_1^2+\|z'\|^2)^{1/2})$ and $f_0 (y_1, z')<1$. By the localization property of the Sibony metric
\cite[Lemma 5]{FL}, it will be enough to prove the property for $D\cap \mathcal U$, and we may restrict the neighborhood
further to reduce ourselves to a model situation.

Since $D$ is bounded, $I^S_D$ is a bounded convex set.
Thus it will be enough to prove that
for any $q$ close enough to $0$, there is a unit vector $X_q$ such that
$\lim_{q \to 0} S_D(q,X_q)=\infty$.
Choose a point $\tilde p \in \partial D$
such that $\|\tilde p -q\|=\d_D(q)$, and $X_q = \|\tilde p -q\|^{-1} (\tilde p -q)$,
which is the outward unit normal at $\tilde p$. We want to have some uniform version of \eqref{model} as $\tilde p$ varies.
Let $\rho (z)= x_1 - f_0 (y_1, z')$ be the (local) defining function of $D$. For $q$ close enough to $0$, $\nabla \rho (\tilde p)$
is close enough to $(1,0)$ so that the signed distance from a point  $z \in \partial D$ to its projection on the tangent hyperplane
$T_{\tilde p} \partial D$ is comparable to
\[
f_0 (y_1, z') - \left( f_0 (\Im \tilde p_1, \tilde p') +
Df_0 (\Im \tilde p_1, \tilde p') \cdot (y_1-\Im \tilde p_1 , z'- \tilde p') \right).
\]
By the Mean Value Theorem, there is some $\theta \in [0,1]$ depending on all the variables involved so that
the above expression equals
\begin{multline*}
\bigg[ Df_0 \big((1-\theta) \Im \tilde p_1 + \theta y_1 ,(1-\theta) \tilde p'+ \theta z'\big)
\\
- Df_0 (\Im \tilde p_1, \tilde p')\bigg]
\cdot (y_1-\Im \tilde p_1 , z'- \tilde p').
\end{multline*}

Since $f_0$ is of class $\mathcal C^1$, its differential is uniformly continuous on any compact set, so there exists some
positive continuous function $f_1 : \R\times \R_+ \rightarrow \R_+$ with $f_1(y,x)=o(|y|+x)$ such that for any $q \in \mathcal U$,
a possibly smaller neighborhood of $0$, if we make a linear change of coordinates so that $\tilde p$ becomes $0$
and the unit outer normal to $\partial D$ at $\tilde p$ becomes $(1,0)$, then the defining function of $\partial D$
becomes $\rho_{\tilde p} (z)=x_1- f_{\tilde p} (y_1, z')$ with $|f_{\tilde p} (y_1, z')| \le f_1 (y_1,\|z'\|)$.
From now on we work with those coordinates.

We can find a function $f=f(y_1,r)$ strictly increasing with respect to $r\in [0,\infty)$
such that $f(y_1,r)=o(|y_1|+r)$ and $D\cap \mathcal U \subset \tilde D$, where
\begin{equation}
\label{dti}
\tilde D:=\left\{ (z_1,z'): |z_1|<1, \|z'\|<1, \mbox{ and } x_1 < f (y_1,\|z'\|) \right\},
\end{equation}
by setting
\[
f(y_1,r):= cr^2 + \max_{0 \le \|z'\| \le r} f_1 (y_1, z'),
\]
with $c>0$ some small constant.

By a slight abuse of notation, let  $f^{-1}$ stand for the inverse function of $r\mapsto f(0,r)$.
Proposition \ref{3} follows from:

\begin{lemma}
\label{finv}
We set, for any $\delta>0$, $q_\delta:=(-\delta,0')$. 
There is a constant $c_1>0$ such that if $\tilde D$ is defined as in \eqref{dti}, with $f(0,0)=0$
and $f$ a continuous function strictly increasing over $\R_+$ in the second variable, then
$$S_{\tilde D}(q_\delta, (1,b')) \ge c_1 \frac{f^{-1}(\delta)}{\delta},\quad b'\in\Bbb C^{n-1}.$$
\end{lemma}

\begin{proof}
The construction is a modification of the proof of \cite[Proposition 3]{FL}.

For any $a\in \R$, let $\varphi(z):= \frac{z-a}{z+a}$. When $a\le 0$, $\varphi$ is holomorphic and bounded
by $1$ in modulus on $\{z\in\C:\Re z<0\}$.  In what follows, the square root of a complex number $z$ is defined
on $\C\setminus \R_-$ with $-\pi<\arg z <\pi$. For $\delta>0$ and $z \in \C\setminus [\delta, \infty)$, we let
$\psi_\delta (z) :=  - (-z+\delta)^{1/2}$. Let $a_\delta:= \psi_\delta (-\delta)= -\delta^{1/2} \sqrt2$,
and $\Phi:= \varphi_{a_\delta} \circ \psi_\delta$. Then $|\Phi(z)| <1$ for $\Re z<\delta$, $\Phi (-\delta)=0$,
$\Phi' (-\delta)= 1/(8\delta)$.

Now we claim  that, for any $\delta>0$ and $\gamma \in (0,1)$,  we can choose a constant $c_\delta $ such that
$u=u_\delta \in A(D,q_\delta)$ where
\begin{eqnarray*}
u_\delta (z_1,z') &:= & \max\left( c_\delta( |\Phi(z_1)|^2 +\|z'\|^2) ,
\|z'\|^{2(1+\gamma)}\right),
\\
 & & \mbox{ for } \|z'\|< f^{-1}(\delta), z \in \tilde D,
\\
u_\delta (z_1,z') &:= & \|z'\|^{2(1+\gamma)},  \mbox{ for } \|z'\|\ge f^{-1}(\delta), z \in \tilde D.
\end{eqnarray*}
We set $c_\delta:=  \frac{f^{-1}(\delta)^{2(1+\gamma)}}{1+f^{-1}(\delta)^2}$.

Clearly, in a neighborhood of $q_\delta$, $u_\delta$ coincides with the first term in the maximum,
so it is locally smooth and, for $\d$ small enough,
\[
\left(\frac{\partial^2 u_\delta}{\partial z_1 \partial \bar z_1}(q_\delta) \right)^{1/2} =
c_\delta^{1/2}\left| \Phi' (-\delta) \right| \ge  \frac{f^{-1}(\delta)^{1+\gamma}}{9 \delta} .
\]

We still need to see that $u_\delta$ is well defined and in the appropriate class. For $\|z'\|< f^{-1}(\delta)$,
if $\Im z_1=0$ and $z\in \tilde D$, then $\Re z_1 < \delta$, so $\Phi(z_1)$ is well defined,
and clearly $0\le u_\delta(z)\le 1$.  In each of the domains where a definition is given,
one easily sees that $\log u_\delta$ is a maximum of
plurisubharmonic functions, so itself plurisubharmonic.

We need to see that both definitions of $u_\delta$ agree in a neighborhood of $\{\|z'\|= f^{-1}(\delta)\}\cap \tilde D$.
It follows from the fact that on that set,
\[
\|z'\|^{2(1+\gamma)}= f^{-1}(\delta)^{2(1+\gamma)} = c_\delta (1 + f^{-1}(\delta)^2)
> c_\delta( |\Phi(z_1)|^2 +\|z'\|^2).
\]

Finally, we obtain $S_{\tilde D}(q_\delta, (1,b')) \ge \frac{f^{-1}(\delta)^{1+\gamma}}{9 \delta}$.
Since this holds for any $\gamma>0$, it follows that
$S_{\tilde D}(q_\delta, (1,b')) \ge \frac{f^{-1}(\delta)}{9 \delta}$.
\end{proof}

\subsection{Proof of Proposition \ref{5}.}
\label{pfberg}

The proof is based on the well-known Ohsawa-Takegoshi
extension theorem which easily reduces the situation to the case $n=1.$

Since $\partial D$ is Lipschitz near $p,$ the uniform exterior cone condition is
satisfied, that is, we may assume that $\Gamma_z:=\{z\}+\Gamma\subset D^c$ for
$z\in\partial D$ near $p,$ where
$\Gamma=\{\z\in\C^n:r>\mbox{Re}(\z_1)>r^{-1}\sqrt{\mbox{Im}^2(\z_1)+||\z'||^2}\},\ r>0.$
For $q$ near $p,$ let $p_q\in\partial D$ be such $||q-p_q||=\d_D(q).$
Set
\[
D_q=\left\{ \zeta \in \C:  (\zeta,q')\in D\right\}.
\]
Denote by $p_{1,q}\in\partial D_q$ and $p_{2,q}\in\partial\Gamma_{p_q}$
the closest points to $q$ on the half-line $q+(\R_+\times \{q'\})$.
Set $\d=||q-p_{1,q}||,$
$\d'=||q-p_{2,q}||$.
Then, since $B(q,\d_D(q))\subset D$ and $D_q \times \{q'\} \cap \Gamma_{p_q} =\emptyset$, a bit of
plane geometry shows that
\begin{equation}\label{a}
\d\le\d'\le \frac{\d_D(q)}{r} \sqrt{1+r^2}.
\end{equation}

We may assume that $p_{1,q}=0,$ $q=(-\d,0')$ and $D_q\subset G_r:=\C\setminus[0,r].$
Using the notations of the proof of Proposition \ref{3},
a conformal mapping from $G_r$ to a punctured unit disk is given by  $\Phi:=\varphi_a \circ \psi \circ \xi$,
where $\xi(z):= \frac{z}{r-z}$, $ \psi(z):= -(-z)^{1/2}$, $a:= \psi \circ \xi(-\d)=\frac{-\d^{1/2}}{(r+\d)^{1/2}}$.
The domain $\Phi(G_r)= \D\setminus \{\frac{-1-a}{-1+a}\}$  has the same Bergman kernel as the disc itself.

The Bergman kernel of $G_r$ is given by $K_{G_r}(z)= |\Phi'(z)|^2 K_{\D}(\Phi(z))$,
where $\D$  stands for the complex unit disc.
Since $\Phi(-\delta)=0$
and $K_{\D}(0)=1/\pi$, we get that
\begin{equation}\label{b}
K_{D_q}(q)\ge K_{G_r}(-\d)=\frac{1}{\pi}\left(\frac{r}{4\d(r+\d)}\right)^2.
\end{equation}

On the other hand, it follows by \cite[Theorem p. 179]{OT} that there exists
a constant $c>0$ depending only on $\mbox{diam}\hskip1pt D$ such that
\begin{equation}\label{c}
K_D\ge c^{n-1}K_{D_q}.
\end{equation}

Now, \eqref{a}, \eqref{b}, and \eqref{c} imply the desired result.
\smallskip

\noindent{\bf Remark.}~Adding an  argument to make the situation uniform, as in the proof
of Proposition \ref{3},
 to the proof of \cite[Proposition 2]{JN},
it follows that if $\partial D$ is of class $\mathcal C^1$ near $p,$ then
$$
\lim_{z\to p}K_{D_q}(z)\d^2_D(z)=\frac{1}{4\pi}\mbox{\ \ \ and hence \ \ }
\liminf_{z\to p}K_D(z)\d^2_D(z)\ge\frac{c^{n-1}}{4\pi}.
$$

\subsection{Proof of Proposition \ref{modelsib}.}
\label{pfsib}

The lower bound  follows easily from  Lemma
\ref{finv} and the boundedness of $D$.

The upper bound is obtained as in the proof of
\cite[Proposition 5]{DNT}:
we use the fact that the indicatrix for the Azukawa metric is pseudoconvex and
contains the indicatrix of the Kobayashi-Royden metric\footnote
{$\kappa_D(z;X)=\inf\{|\alpha|:\exists\varphi\in\mathcal O(\D,D) \hbox{
with }\varphi(0)=z,\alpha\varphi'(0)=X\}.$}. We will show that a
certain Hartogs figure is included in that latter indicatrix.

Let  $\B^d$ stand for the unit ball of $\C^d$.  By a linear change of variables,
we may assume that $\{-\d\} \times \overline\B^{n-1} \subset D$, so for any $X\in \overline\B^{n-1}$,
$\varphi (\zeta):= q_\d + \zeta X$ provides a map  from $\D$ to $D$ with $\varphi'(0)=X$.
If we show that there exists $c_1>0$ such that
any vector $X \in c_1 \frac\d{f^{-1}(\d)} \D \times \partial \B^{n-1}$ can be realized
as the derivative at the origin of a holomorphic map $\varphi$ from $\D$ to $D$ with $\varphi(0)=q_\d$,
then the whole of $c_1 \frac\d{f^{-1}(\d)} \D \times  \B^{n-1}$
will be included in $I^S_D(q_\d).$

Let
\[
\varphi(\zeta):= q_\d + \zeta \left( c_1 \frac\d{f^{-1}(\d)} , X' \right) \mbox{ where } X' \in  \partial \B^{n-1}.
\]
Let $r(z)=\Re z_1 + O(|\Im z_1|) - f(\|z'\|)$ be the defining function of $D$. Then, for $c_1$ well chosen,
\begin{equation}
\label{insmall}
r\left( \varphi(\zeta) \right) \le - \d +  \frac\d{f^{-1}(\d)} |\zeta| - f(  |\zeta| ) \le  - f(  |\zeta| )  \le 0
\end{equation}
when $|\zeta| \le f^{-1}(\d)$.  On the other hand, when $|\zeta| \ge f^{-1}(\d)$, the hypothesis on $f$
implies $f(  |\zeta| ) \ge \frac\d{f^{-1}(\d)} |\zeta|$ and we get $r\left( \varphi(\zeta) \right) \le 0$ again.
\vskip4pt

\subsection{Proof of Proposition  \ref{compconv}.}
\label{pfcomp}

We follow the strategy of the proofs of Proposition \ref{3} and Lemma \ref{finv}.
Take coordinates so that
the complex line through $q$ and $p_q$ be the $z_1$ axis. An argument of uniformity similar to the one
at the beginning of the proof of Proposition \ref{3} shows that we can choose a domain $D_q$ locally
defined by
\[
\Re z_1 < f_0 (\Im z_1, z'),
\]
with $f_0(0,0)=0$, $Df_0 (0,0)=0$,
$D\subset D_q$, and the growth of $f_0$ is uniformly controlled in a neighborhood of the original point $p$.
Let
\begin{equation}
\label{risingsun}
f_1(y,r):= \max_{\|w'\| \le r} f_0(y,w').
\end{equation}
Then $f_1$ is continuous, nonnegative, increasing in $r$. Finally, for any $\eta>0$, let $g_\eta(y,r):=f_1(y,r)+\eta r^2$,
and let $D^\eta$ be the domain defined locally by
\[
\Re z_1 < g_\eta (\Im z_1, \|z'\|).
\]
Note that $\delta_D(q)=\delta_{D^\eta}(q)$.
Let $g_\eta^{-1}$ be the inverse function of $r\mapsto g_\eta(0,r)$. 
 By Lemma \ref{finv},
 \[
 cS_D(q, (1,b'))\ge cS_{D^\eta}(q, (1,b')) \ge \frac{g_\eta^{-1}(\delta_{D}(q))}{\delta_D(q)}.
 \]
Let $z'(\eta)$ be a point such that $g_\eta (0, \|z'(\eta)\|)=\d_D(q)$. Letting $\eta\to 0$, by compactness
there is a cluster point $z'(0)$ such that
\[
\|z'(0)\| = \lim_{\eta\to0} \|z'(\eta)\| = \lim_{\eta\to0} g_\eta^{-1}(\delta_{D}(q)),
\]
so that
\[
cS_D(q, (1,b'))\ge \frac{\|z'(0)\|}{\delta_D(q)},
\]
and
$
f_1(0,\|z'(0)\|) =\lim_{\eta\to0} g_\eta(0,\|z'(\eta)\|)= \delta_D(q),
$
so $(\delta_D(q),z'(0)) \in \partial D_q$ and
by minimality of $\delta_{D_q}^*$ we have
\[
 \|z'(0)\| \ge  \delta_{D_q}^*(q) \ge \delta_D^*(q).
\]
On the other hand, in any direction the Sibony metric is bounded below by the Euclidean norm
because the domain is bounded. Taking the maximum of the estimates, we obtain the desired result.



{}

\end{document}